\documentclass[11pt]{amsart}
\usepackage{amsmath, amsthm, mathabx,amscd, amsfonts, amssymb, hyperref, mathrsfs, textcomp, epsfig, a4wide, graphicx, IEEEtrantools, tikz, verbatim, xypic }

\usepackage[all,cmtip]{xy}
\usetikzlibrary{matrix, decorations.pathreplacing,angles,quotes,cd}

\newtheorem{thm}{Theorem}[section]

\newtheorem{prop}[thm]{Proposition}
\newtheorem{lemma}[thm]{Lemma}

\theoremstyle{definition}
\newtheorem{defn}[thm]{Definition}

\usepackage{mathtools}

\newcommand{\T}{\mathcal{T}}
\newcommand{\Z}{\mathbb{Z}}
\newcommand{\spin}{\mathfrak{s}}
\newcommand{\bardel}{\overline{\partial}}

     \RequirePackage{rotating}                   
    \def\HMt{%
       \setbox0=\hbox{$\widehat{\mathit{HM}}$}
       \setbox1=\hbox{$\mathit{HM}$}
       \dimen0=1.1\ht0
       \advance\dimen0 by 1.17\ht1
       \smash{\mskip2mu\raise\dimen0\rlap{%
          \begin{turn}{180}
              {$\widehat{\phantom{\mathit{HM}}}$}
           \end{turn}} \mskip-2mu    
                \mathit{HM}
    }{\vphantom{\widehat{\mathit{HM}}}}{}}
    
         \RequirePackage{rotating}                   
    \def\HSt{%
       \setbox0=\hbox{$\widehat{\mathit{HS}}$}
       \setbox1=\hbox{$\mathit{HS}$}
       \dimen0=1.1\ht0
       \advance\dimen0 by 1.17\ht1
       \smash{\mskip2mu\raise\dimen0\rlap{%
          \begin{turn}{180}
              {$\widehat{\phantom{\mathit{HS}}}$}
           \end{turn}} \mskip-2mu    
                \mathit{HS}
    }{\vphantom{\widehat{\mathit{HS}}}}{}}

\theoremstyle{remark}
\newtheorem{remark}{Remark}
\newtheorem{example}{Example}

\begin{document}
\title{Monopole Floer homology and invariant theta characteristics}

\author{Francesco Lin}
\address{Department of Mathematics, Columbia University} 
\email{flin@math.columbia.edu}

\begin{abstract}
We describe a relationship between the monopole Floer homology of three-manifolds and the geometry of Riemann surfaces. Consider an automorphism $\varphi$ of a compact Riemann surface $\Sigma$ with quotient $\mathbb{P}^1$. There is a natural correspondence between theta characteristics $L$ on $\Sigma$ which are invariant under $\varphi$ and self-conjugate spin$^c$ structures $\spin_L$ on the mapping torus $M_{\varphi}$ of $\varphi$. We show that the monopole Floer homology groups of $(M_{\varphi},\spin_L)$ are explicitly determined by the eigenvalues of the (lift of the) action of $\varphi$ on $H^0(L)$, the space of holomorphic sections of $L$. Decategorifying our computation, we also obtain that the dimension of $H^0(L)$ equals the Reidemeister-Turaev torsion of $(M_{\varphi},\spin_L)$. Finally, we combine our description with the Atiyah-Bott $G$-spin theorem to provide explicit computations of the Floer homology groups for all automorphisms $\varphi$ of prime order in terms of ramification data.
\end{abstract}

\maketitle

\section*{Introduction}
The interplay between Seiberg-Witten theory and complex geometry has been very fruitful since the early development of the theory (see for example \cite{Mor},\cite{FM},\cite{Bru} among the many). Following this thread, in this paper we investigate the interactions between monopole Floer homology \cite{KM}, a package of topological invariants of three-manifolds equipped with a spin$^c$ structure $(Y,\spin)$, and the classical study of the geometry of algebraic curves.
\par
Consider a compact Riemann surface $\Sigma$ of genus $g\geq 2$, and choose an automorphism $\varphi$ (necessarily of finite order $d$) such that the quotient is $\mathbb{P}^1$. By Kummer theory, we can think of $\Sigma$ as the Riemann surface associated to the affine curve $z^d=p(w)$ for some polynomial $p$ (which possibly has multiple roots), and $\varphi(z,w)=(\zeta z,w)$ for some primitive $d$th root of unity (some authors call such curves \textit{$\mathbb{Z}_d$-curves} \cite{FZ}). We then consider two objects:
\begin{itemize}
\item From the point of view of low-dimensional topology, we consider the \textit{mapping torus} $M_{\varphi}$ of $\varphi$, which is a three-manifold with $b_1=1$;
\item On the algebraic geometry side, we consider a \textit{theta characteristic} of $\Sigma$, i.e. one of the $2^{2g}$ holomorphic line bundles $L$ which are square roots of the canonical bundle $K$, i.e. $L^2\cong K$ \cite{ACGH}.
\end{itemize}
The starting point of our result is that because $\Sigma/\varphi=\mathbb{P}^1$, there is a natural bijection
\begin{equation*}
\{\text{self-conjugate spin$^c$ structures $\spin$ on $M_{\varphi}$}\}\longleftrightarrow\{\varphi\text{-invariant theta characteristics }L\}
\end{equation*}
where a spin$^c$ structure $\spin$ is \textit{self-conjugate} if it is isomorphic to its conjugate $\spin\cong\bar{\spin}$. Equivalently, $\spin$ is induced by some spin structure (as $b_1(M_{\varphi})=1$, it is induced by exactly two of them). We will denote the spin$^c$ structure corresponding to $L$ by $\spin_L$. This is closely related to the classical observation due to Atiyah \cite{Ati} that there is a natural bijection
\begin{equation}\label{atiobs}
\{\text{spin structures on $\Sigma$}\}\longleftrightarrow\{\text{theta characteristics }L\}.
\end{equation}
We will investigate the interplay between the monopole Floer homology groups of $(M_{\varphi},\spin_L)$ and the geometry of the theta characteristic $L$, specifically the space of its holomorphic sections $H^0(L)$. The latter is very well-studied object closely related to the study of bitangents in projective geometry \cite[Chapter 5]{Dol}; in general, its dimension $h^0(L)$ depends on the underlying holomorphic geometry of $\Sigma$ in a subtle way.
In our setup, because $L$ is $\varphi$-invariant, the action of $\varphi$ on $\Sigma$ lifts to an action $\tilde{\varphi}$ on $L$ (whose order is either $d$ or $2d$) compatible with the identification $L^2\cong K$, and we obtain an induced finite order action $\tilde{\varphi}^*$ on $H^0(L)$ by pull-back. Our main result is the following.
\begin{thm}\label{main}
Suppose $\Sigma/\varphi=\mathbb{P}^1$, and consider a $\varphi$-invariant theta characteristic $L$. Then, for suitable choices of metric and perturbation, the Floer chain complex of $(M_{\varphi},\spin_L)$ (possibly with coefficients in a local system $\Gamma$) is explicitly determined by the eigenvalues of the action of $\tilde{\varphi}^*$ on $H^0(L)$. Hence for any local coefficient system $\Gamma$ one can compute explicitly both the Floer homology group $\HMt_{\bullet}(M_{\varphi},\spin_L;\Gamma)$ and the $U$-action in terms of the geometry of the theta characteristic $L$.
\end{thm}

\begin{remark}
Because $\spin_L$ is a torsion spin$^c$ structure, the associated Floer homology groups admit an absolute $\mathbb{Q}$-grading \cite[Chapter 38]{KM}; the description we will provide of the Floer chain complex allows to determine this grading up to an overall shift.
\end{remark}

The exact description of the Floer chain complex will require some additional setup, and we refer the reader to Section \ref{chain} for a precise statement. For now we will content ourselves with some homological results that directly follow from the description. As $b_1(M_{\varphi})=1$, there are interesting local coefficient systems on the configuration space. Recall for example that to each a $1$-cycle $\eta$ in $M_{\varphi}$, in \cite[Section 3.7]{KM} the authors associate a local system $\Gamma_{\eta}$ with fiber $\mathbb{R}$. The corresponding Floer groups with local coefficients $\HMt_{\bullet}(M_{\varphi},\spin_L;\Gamma_\eta)$ are $\mathbb{R}[U]$-modules, finite dimensional as $\mathbb{R}$-vector spaces, which play a central role when discussing gluing formulas for Seiberg-Witten invariants \cite[Section 3.8]{KM}. We will begin by stating the result for such local systems $\Gamma_{\eta}$ as they are simpler to state.
\par
We will see in Section \ref{prelim} that, because $\Sigma/\varphi=\mathbb{P}^1$, $\zeta$ and $\bar{\zeta}$ are never simultaneously eigenvalues of $\tilde{\varphi}^*$, so that we can write the list of eigenvalues (which are elements in $S^1$) as
\begin{equation}
e^{ i\pi r_1},\dots, e^{ i\pi r_k}\quad\text{with}\quad0<|r_1|<\dots<|r_k|<1,
\end{equation}
where $e^{ i\pi r_i}$ has multiplicity by $m_i$.
\begin{thm}\label{thmlocal}
Suppose $\Sigma/\varphi=\mathbb{P}^1$ and $[\eta]\neq0\in H^1(M_{\varphi};\mathbb{R})$. Then
\begin{equation*}\mathrm{dim}_{\mathbb{R}}\HMt_{\bullet}(M_{\varphi},\spin_L;\Gamma_\eta)=h^0(L).
\end{equation*}
Furthermore, denote by $c_L$ the number of changes of sign in the sequence $r_1,\dots, r_k$; we set $c_L=-1$ if such sequence is empty. Then $\HMt_{\bullet}(M_{\varphi},\spin_L;\Gamma_\eta)$ has rank $c_L+1$ as a $\mathbb{R}[U]$-module. 
\end{thm}
\begin{remark}\label{Z2grad}
The Floer homology groups admit a canonical $\mathbb{Z}/2$-grading \cite[Section $22.4$]{KM}, and the group in the statement of the theorem lies entirely in even grading.
\end{remark}

We then discuss the `basic' version of Floer homology $\HMt_{\bullet}(M_{\varphi},\spin_L)$; this is especially useful for example when studying the topology of negative definite manifolds bounding $M_{\varphi}$ via the analogue of Fr\o yshov invariants for manifolds with $b_1=1$ \cite{OSd}. Denoting the tower $\T_d=\mathbb{Z}[U,U^{-1}]/U\cdot\mathbb{Z}[U]$ with bottom element $1$ lying in degree $d$, we have the following.
\begin{thm}\label{thmnonlocal}
Suppose $\Sigma/\varphi=\mathbb{P}^1$. Then, up to an overall grading shift, we have the decomposition of $\mathbb{Z}[U]$-modules
\begin{equation}
\HMt_{\bullet}(M_{\varphi},\spin_L)=\T_{-1}\oplus \T_{ -2\Delta_L} \oplus \mathit{HM}_{\bullet}(M_{\varphi},\spin_L)
\end{equation}
where:
\begin{itemize}
\item $\Delta_L\geq0$ is the difference between the maximum and the minimum in the sequence $\sum_{i=1}^n \mathrm{sgn}(r_i)m_i$, $n=1,\dots, k$.
\item the action of a generator of $H^1(M_{\varphi};\mathbb{Z})$ maps the second tower (which lies in even absolute $\mathbb{Z}_2$ grading) surjectively onto the first.
\item the reduced Floer homology group $\mathit{HM}_{\bullet}(M_{\varphi},\spin_L)$ has rank as a $\mathbb{Z}[U]$-module $\max\{c_L,0\}$ and lies in even absolute $\mathbb{Z}_2$ grading.
\end{itemize}
In particular, the reduced Floer homology $\mathit{HM}_{\bullet}(M_{\varphi},\spin_L)$ is trivial if and only if all eigenvalues lie either on the upper or lower half plane.
\end{thm}

Regarding the proof of Theorem \ref{main}, the key complication is that the spin$^c$ structures $\spin_L$ are torsion. This implies that even though it is true for a general automorphism $\varphi$ of a Riemann surface that the equations of $M_{\varphi}$ do not admit irreducible solutions for specific metrics and perturbations \cite{Lop},\cite{MOY}, the unperturbed Seiberg-Witten equations always admit a $b_1(M_{\varphi})$-dimensional torus $\mathbb{T}$ of reducible solutions \cite[Section $4.2$]{KM}, which turns out to be \textit{highly degenerate}: for example, in the case $\varphi$ is the identity, the Hessian of the Chern-Simons-Dirac functional is singular along a copy of the ($2g-2$)-dimensional theta divisor of $\Sigma$ in the ($2g+1$)-dimensional torus of reducible solutions on $M_{\varphi}=S^1\times \Sigma$; the computation of its Floer homology via surgery techniques \cite{JM} indeed confirms that very delicate things must happen when trying to perturb the equations to achieve transversality.
\par
In this paper we will focus on the case $\Sigma/\varphi=\mathbb{P}^1$, so that $M_{\varphi}$ has $b_1=1$ and the singularities on the circle of reducible solutions $\mathbb{T}$ are as mild as one could hope; we will see how to achieve transversality while still maintaining a concrete understanding of the underlying geometry by adding suitable perturbations. In order to do so, we will use the same pertubations as those employed in \cite[Chapter 37]{KM} to study the genus one case; on the other hand, our arguments are significantly more delicate (and indeed generalize those appearing there) as we cannot rely on the existence of a flat metric on the mapping torus.

\begin{remark}
From the point of view of geometrization, our manifolds of interest correspond to the product geometry $\mathbb{H}^2\times\mathbb{R}$; three-manifolds with this geometry naturally arise when investigating the interplay of scalar curvature and \textit{non}-torsion spin$^c$ structures \cite{IY}.
\end{remark}

\subsection*{Concrete examples. }These results allow for direct computations of the Floer homology groups in many concrete examples where we can access the geometry of $L$. We begin by discussing the case of the hyperelliptic involution; this leaves all spin $2^{2g}$ structures invariant (and in fact is the only such automorphism \cite{KS}). We determine the Floer homology for all self-conjugate spin$^c$ structures as follows.
\begin{prop}\label{order2}
Consider a hyperelliptic Riemann surface of genus $g$ with hyperelliptic involution $\varphi$, and $[\eta]\neq0$. Among the $2^{2g}$ self-conjugate spin$^c$ structures $\spin_L$:
\begin{itemize}
\item exactly $\binom{2g+1}{g}$ have vanishing $\HMt_{\bullet}(M_{\varphi},\spin_L;\Gamma_\eta)$;
\item for any $1\leq w\leq g$ odd, there are exactly $\binom{2g+2}{g-w}$ self-conjugate spin$^c$ structures with $\HMt_{\bullet}(M_{\varphi},\spin_L;\Gamma_\eta)= \mathbb{R}[U]/U^{(w+1)/2}$.
\end{itemize}
\end{prop}

The proof of this result essentially follows, via our previous results, from the determination of the dimension of $h^0(L)$ for all theta characteristic on an hyperelliptic curve; this is a classical computation in algebraic geometry done via the theory of Weierstrass points \cite{ACGH}, \cite{BS}.
\\
\par
We then focus our attention on the case of automorphism of prime odd order. In this case, there is only one invariant theta characteristic $L_0$ \cite{KS}, and we denote the corresponding self-conjugate spin$^c$ structure by $\spin_0$. We have the following concrete computation.
\begin{thm}\label{compprime}
Consider an automorphism $\varphi$ of $\Sigma$ with prime odd order with $\Sigma/\varphi=\mathbb{P}^1$. Then the spectrum of the action of $\tilde{\varphi}^*$ on $H^0(L_0)$ can be explicitly determined in terms of rotation angles $d\varphi_p\in S^1$ at the fixed points $p$ of $\varphi$. As a consequence, the Floer homology groups of $(M_{\varphi},\spin_0)$ can be determined explicitly in terms of the ramification data of $\varphi$.
\end{thm}

Again, the description of the explicit procedure requires some additional setup; let us point out some general consequences, starting from the case of automorphisms of order $3$.
\begin{prop}\label{order3}
Consider an automorphism $\varphi$ of $\Sigma$ with order $3$ with $\Sigma/\varphi=\mathbb{P}^1$. Denote by $n_{\pm}$ the number of fixed points $p$ at which $d\varphi_p=e^{\pm 2\pi i/3}$ respectively. Then for $[\eta]\neq0$ we have
\begin{equation*}
\HMt_{\bullet}(M_{\varphi},\spin_0;\Gamma_\eta)= \mathbb{R}[U]/U^{n}.
\end{equation*}
where $n=|n_+-n_-|/3$.
\end{prop}

In all concrete examples we have seen so far the Floer group with twisted coefficient is cyclic; going only slightly further, we find examples that show that this is not always the case.

\begin{prop}\label{order5}
There exist automorphisms $\varphi$ of order $5$ for which $\HMt_{\bullet}(M_{\varphi},\spin_0;\Gamma_\eta)$ is not a cyclic $\mathbb{R}[U]$-module; in fact, for every $n,m\geq 1$ there exists an automorphism with
\begin{equation*}
\HMt_{\bullet}(M_{\varphi},\spin_0;\Gamma_\eta)= \mathbb{R}[U]/U^{n}\oplus \mathbb{R}[U]/U^{m}
\end{equation*}
where we can require that either the top generators or the bottom generators of the summand lie in the same absolute grading.
\end{prop}
\begin{remark}
We will see that in the case of order $5$ maps, this exhaust all possibilities allowed; more generally, looking at larger primes $p$, our method allows to find automorphisms with arbitrarily large rank as an $\mathbb{R}[U]$-module.
\end{remark}
The key observation behind Theorem \ref{compprime} is that under the correspondence (\ref{atiobs}), the space of holomorphic sections of $L$ corresponds to the space of harmonic spinors (the kernel of the Dirac operator). One can then infer information about the action of $\tilde{\varphi}^*$ using the Atiyah-Bott $G$-spin theorem \cite{AB}, which relates the trace of the action to local information at the fixed points. While this is in general not sufficient to compute the whole spectrum of $\tilde{\varphi}^*$, it turns out that under our assumptions (prime order with quotient $\mathbb{P}^1$) it is.

\vspace{0.3cm}
\subsection*{Computations in complex and symplectic geometry} We have seen how one can use the understanding of the geometry of curves and their automorphisms to compute the corresponding topological invariants of three-manifolds; we now discuss some instances in which the opposite procedure leads to interesting results.
\\
\par
\textit{Computations of $H^0(L)$.} It is well known that $h^0(L)=\mathrm{dim}_{\mathbb{C}}H^0(L)$ is not invariant under deformations of the complex structure on $\Sigma$ \cite{Hit}. On the other hand, an amusing consequence of Theorem \ref{main} is a `topological' proof that for a $\varphi$-invariant theta characteristic $L$, $h^0(L)$ is invariant under $\varphi$-invariant deformations of the complex structure. As mentioned earlier, in general the $G$-spin theorem is not enough to compute the action of $\tilde{\varphi}^*$ on $H^0(L)$, and in fact not even its dimension. One can instead use Theorem \ref{main} to obtain information about $H^0(L)$ and the action of $\tilde{\varphi}^*$ on it, provided one can compute the relevant Floer homology groups from other topological manners.
\par
Let us first point out that our results imply that $h^0(L)$ can be computed purely in terms of `classical' topological invariants.
\begin{thm}\label{turaev}
Suppose $\Sigma/\varphi=\mathbb{P}^1$, and consider a $\varphi$-invariant theta characteristic $L$. Then $h^0(L)$ is the Reidemeister-Turaev torsion $\tau(M_{\varphi},\spin_L)$.
\end{thm}
A concrete computational method for the Reidemeister-Turaev torsion of mapping tori (in terms of Fox derivatives) can be found in \cite[Section VII$.5.2$]{TurBook}. Theorem \ref{turaev} follows because Theorem \ref{thmlocal} and Remark \ref{Z2grad} also hold if we replace $\Gamma_{\eta}$ with suitable Novikov-type local coefficient systems $\Gamma$ as introduced in \cite[Section $30.2$]{KM}. We then have
\begin{equation*}
h^0(L)=\chi\left(\HMt_{\bullet}(M_{\varphi},\spin_L;\Gamma)\right).
\end{equation*}
The latter group is isomorphic to the version with a non-exact perturbation $\HMt_{\bullet}(M_{\varphi},\spin_L,c;\Gamma)$ (cf. the discussion in \cite[Section $42.5$]{KM}) whose Euler characteristic is well-known to be the same as the corresponding Reidemeister-Turaev torsion invariant \cite{MT}, \cite{Tur} (see also \cite{Nic}). 
\begin{example}Consider the torus knot $T_{p,q}$, where $p,q$ are coprime integers. The zero surgery $S^3_0(T_{p,q})$ is the mapping torus of an automorphism $\varphi$ of order $pq$ of a surface of genus $(p-1)(q-1)/2$, leaving fixed a unique theta characteristic $L_0$ (because $H_1(S^3_0(T_{p,q});\mathbb{Z})=\mathbb{Z}$). Writing the symmetrized Alexander polynomial of $T_{p,q}$ as
\begin{equation*}
\Delta_{T_{p,q}}(T)=T^{(p-1)(q-1)/2}\frac{(1-T)(1-T^{pq})}{(1-T^p)(1-T^q)}=a_0+\sum_{i>0}a_i(T^i+T^{-i}),
\end{equation*}
Theorem \ref{turaev} then implies the relation
\begin{equation*}
h^0(L_0)=\sum_{i>0}ia_i.
\end{equation*}
Notice that in the case of the torus knot $T_{2,3}$ the $0$-surgery (for which $h^0(L_0)=1$) is one of the flat manifolds discussed in \cite{KM}.
\end{example}
To go beyond the computation of $h^0(L)$, notice that the arguments in \cite{OSd}, which are based on the surgery exact triangle \cite{KMOS}, also imply that $\HMt_{\bullet}(S^3_0(T_{p,q}),\spin_{L_0};\Gamma_\eta)$ is cyclic as a $\mathbb{R}[U]$-module. We can therefore conclude from our theorem that all the eigenvalues of the action of $\tilde{\varphi}^*$ on $H^0(L_0)$ lie either in the upper or lower half-plane. More generally, one can in principle combine the very flexible computational tools of Heegaard Floer homology together with its isomorphism with monopole Floer homology (see \cite{KLT}, \cite{CGH}, and subsequent papers) to compute information about the action of $\tilde{\varphi}^*$ on $H^0(L)$ in other families of explicit examples.
\\
\par
\textit{Relation with periodic Floer homology. }Hutchings and Thaddeus (see \cite{HS}) defined the periodic Floer homology groups of an area preserving diffeomorphism of a surface; this construction, inspired by embedded contact homology, looks at unions of finite orbits of the map (unlike the usual fixed point Floer homology, which is defined in all dimensions, and only looks at fixed points). It was shown by Lee and Taubes \cite{LC} that this symplectic invariant is equivalent to the monopole Floer homology of the mapping torus $M_{\varphi}$, so that our main result implies computation of certain cases of such invariants. The most notable feature is that because we are dealing with torsion spin$^c$ structures $\spin_L$, the corresponding construction on the symplectic side always involves the \textit{non}-monotone (and indeed Calabi-Yau) case; on the monopole Floer homology side, this means that one needs again to work with a Novikov-type local system $\Gamma$, and appeal to the isomorphism for $\spin$ torsion $\HMt_*(Y,\spin,\Gamma)\cong \HMt_*(Y,\spin, c;\Gamma)$ where $c\neq 0$ is the class of a non-exact perturbation \cite[Theorem  $31.1.3$]{KM}.
\par
It would be interesting to find a direct link between the symplectic and algebraic geometry for such finite order classes; of course a subtlety is that a finite order automorphism is very far from having non-degenerate periodic orbits in the sense of \cite{HS}.

\vspace{0.3cm}
\subsection*{Structure of the paper. } In Section \ref{prelim} we recall the basic concepts relating spin structures and theta characteristic on Riemann surfaces with symmetries, and discuss the special features of having quotient $\mathbb{P}^1$. Section \ref{perturb} represents the technical core of the paper; we show that we can suitably perturb the equations while still maintaining a solid understanding of the geometry of the problem. From this, in Section \ref{chain} we provide a concrete description of the Floer chain complexes in terms of the spectrum of the action on $H^0(L)$, and discuss computations. Finally, in Section \ref{comp} we provide concrete computations for all automorphisms of prime order.

\vspace{0.3cm}
\subsection*{Acknowledgements. }
The author was partially supported by the Alfred P. Sloan foundation and NSF grant DMS-1948820.

\vspace{0.5cm}

\section{Preliminaries}\label{prelim}
Throughout the paper, it will be convenient to switch between holomorphic and spin geometry; we discuss some background on this dictionary here (See for example \cite{Hit}, \cite{Ati}, \cite{Roe} for more details).
\par
Consider a $2n$-dimensional Riemannian manifold $(X,g)$ equipped with a spin structure $\spin$. To this, we can associate a pair of spinor bundles $S^{\pm}\rightarrow X$ (which are $2^n$-dimensional hermitian bundles), and the corresponding (chiral) Dirac operators
\begin{equation*}
D^{\pm}:S^{\pm}\rightarrow S^{\mp}.
\end{equation*}
In the case $X$ is a Riemann surface $\Sigma$ of genus $g$, this construction can be interpreted purely in terms of K\"ahler geometry; this is because $\mathrm{SO}(2)=\mathrm{U}(1)$, and $\mathrm{Spin}(2)\rightarrow \mathrm{SO}(2)$ corresponds to the double cover of $\mathrm{U}(1)$ onto itself. Namely, a spin structure is the same as a choice of one of the $2^{2g}$ square roots $L$ of the canonical bundle (i.e. $L^2=K$); these are classically known as theta characterstics. One can then identify $S^+=L$ and $S^-=\bar{L}$ with the natural connections and the Dirac operator $D^+$ can be interpreted as
\begin{equation}\label{dirackahler}
D^+=\sqrt{2}\bar{\partial}: \Gamma(L)\rightarrow \Gamma(\bar{K}\otimes L)=\Gamma(\bar{L})
\end{equation}
where we identify $\bar{K}\otimes L=\bar{L}$.
\begin{remark}Throughout the paper, we use the standard conventions for complex geometry (see for example \cite{Huy}) so that $\bar{\partial}f=(df)^{0,1}$ is given by
\begin{equation*}
\bar{\partial} f=\frac{\partial f}{\partial \bar{z}}d\bar{z}\quad\text{with}\quad \frac{\partial f}{\partial \bar{z}}=\frac{1}{2}\left(\frac{d}{dx}+i\frac{d}{dy}\right).
\end{equation*}
The factor of $\sqrt{2}$ appears in (\ref{dirackahler}) because in local coordinates the Dirac operator is $d/dx+id/dy$, and $d\bar{z}$ has norm $\sqrt{2}$.
\end{remark}

The space of holomorphic sections $H^0(L)$ can be identified as the space of positive harmonic spinors $\mathcal{H}^+$. Similarly, $D^-$ is interpreted as the adjoint operator
\begin{equation*}
D^-=\sqrt{2}\bar{\partial}^*: \Gamma(\bar{L})=\Gamma(\bar{K}\otimes L)\rightarrow\Gamma({L}),
\end{equation*}
which gives us the identification of $H^1(L)$ with the space of negative harmonic spinors $\mathcal{H}^-$.
\\
\par
These ideas were used in \cite{Hit} to provide the first examples of manifolds where the dimension of space of harmonic spinors depends on the underlying metric. There, it is also shown that in general the dimension of $\mathcal{H}^{\pm}$ only depends on the conformal class of the metric, hence in the two-dimensional case only on the Riemann surface structure. On the other hand, the parity of $h^0(L)$ is independent of the underlying holomorphic structure (a theta characteristic is called even/odd accordingly). This was first shown by Mumford \cite{Mum}, and reinterpreted via mod $2$ index theory in the language we have just discussed by Atiyah \cite{Ati}.
\\
\par
Consider now a holomorphic automorphism $\varphi$ of $\Sigma$. We will be mostly be interested in the case when $g\geq 2$, so that $\varphi$ automatically has finite order; if the case of a surface of genus $1$ (which corresponds to the computations of flat three-manifolds in \cite{KM}) we will take this as an assumption. Then $\varphi$ acts on the set of spin structures on $\Sigma$ by pull-back with at least one fixed point \cite[Lemma $5.1$]{Ati}. Given such a fixed point $\spin$, $\varphi$ lifts to an automorphism of the spin structure $\tilde{\varphi}$; if $\varphi$ has order $e$, then $\tilde{\varphi}^e$ is a lift of the identity hence it is either $\pm \mathrm{id}$. This leads in general dimensions to a quite rich story, see \cite{AB}, \cite{Sha}; in the case of surfaces, the situation is quite simple, and we record some results we will use later.
\begin{lemma}\label{liftorder}
Suppose $\varphi$ fixes a spin structure $\spin$ on $\Sigma$.
\begin{itemize}
\item If the order of $\varphi$ is odd, then $\varphi$ admits a lift of the same order;
\item if the order of $\varphi$ is even, say $2m$, and $\varphi^m$ has at least one fixed point, then any lift $\varphi$ has order $4m$. 
\end{itemize}
In particular, the hyperelliptic involution lifts to a spin automorphism of order $4$, cf. \cite[Proposition 8.46]{AB}.
\end{lemma}
\begin{proof}
If the order $e$ is odd, then exactly one of $\tilde{\varphi}$ and $-\tilde{\varphi}$ has order $e$. Suppose then $e=2m$ and $\varphi^m$ has a fixed point. Around this point, $\varphi^m$ acts as a rotation by $\pi$; this implies that any lift $\tilde{\varphi}^m$ acts at that point as $\pm i$, so that it has order $4$, and $\tilde{\varphi}$ has order $4m$.
\end{proof}

\begin{example}\label{torus}
Consider for example the torus obtained by identifying the opposite side of a square with the product Lie spin structure; in this case the spinor bundles $S^{\pm}$ are trivial, and we can lift the actions
\begin{equation*}
z\mapsto -z\quad \text{and} \quad z\mapsto iz
\end{equation*}
of order $2$ and $4$ respectively as
\begin{align*}
(z,w_+,w_-)&\mapsto (-z,iw_+,-iw_-)\\
(z,w_+,w_-)&\mapsto (iz,e^{\pi i/4}w_+,e^{-\pi i/4}w_-).
\end{align*}
Similarly the actions of order $3$ and $6$ can be studied via the realization of the torus as the quotient of a regular hexagon.
\end{example}

\begin{remark}\label{conj}
It is true in general that the action of $\varphi$ on $S^{\pm}$ (which are conjugate bundles) are conjugate.
\end{remark}

Consider then the action of $\tilde{\varphi}$ on the space of holomorphic sections $H^0(L)$ or, equivalently, the space of positive harmonic spinors $\mathcal{H}^+$. Because $\tilde{\varphi}$ has finite order $d$, this action is diagonalizable with eigenvalues $d$-th roots of unity, and have the following.
\begin{defn}
The spectrum $\mathrm{spec}(\tilde{\varphi},\spin)$ is the set of eigenvalues of the action of $\tilde{\varphi}$ on $H^0(L)$, counted with multiplicity.
\end{defn}

A key observation for our analysis of transversality is the following.

\begin{lemma}\label{noconj}
Suppose the quotient of $\Sigma$ by $\varphi$ is $\mathbb{P}^1$. Then
\begin{equation*}
\mathrm{spec}(\tilde{\varphi},\spin)\cap \overline{\mathrm{spec}(\tilde{\varphi},\spin)}=\emptyset.
\end{equation*}
In particular, $\pm1$ do not belong to $\mathrm{spec}(\tilde{\varphi},\spin)$.
\end{lemma}
In light of this lemma, we will denote the eigenvalues of $\tilde{\varphi}^*$ as
\begin{equation}
e^{ \pi i r_1},\dots ,e^{ \pi i r_k}\quad\text{with}\quad0<|r_1|<\dots<|r_k|<1,
\end{equation}
and $e^{ \pi ir_j}$ has multiplicity by $m_j$.
\begin{proof}
Consider the (complex linear) multiplication map
\begin{equation*}
m:H^0(L)\otimes H^0(L)\rightarrow H^0(K)
\end{equation*}
given by the identification $L^2=K$. This is compatible with the actions of $\varphi$ and its lift $\tilde{\varphi}$ in the sense that
\begin{equation*}
m(\tilde{\varphi}^*(s_1),\tilde{\varphi}^*(s_2))=\varphi^*(m(s_1,s_2))
\end{equation*}
Suppose by contradiction that both $\zeta,\bar{\zeta}\in S^1$ belong to $\mathrm{spec}(\tilde{\varphi},\spin)$, and consider non-zero sections with
\begin{equation*}
\tilde{\varphi}^*(s_1)=\zeta\cdot s_1\qquad \tilde{\varphi}^*(s_2)=\bar{\zeta}\cdot s_2.
\end{equation*}
Then we obtain
\begin{equation*}
\varphi^*(m(s_1,s_2))=m(s_1,s_2),
\end{equation*}
so that $m(s_1,s_2)$ is (by unique continuation) a non-zero $\varphi$-invariant holomorphic $1$-form on $\Sigma$. But this cannot exist because such a form would descend to a non-zero holomorphic $1$-form on $\mathbb{P}^1$.
\end{proof}

\begin{example}\label{hyperex} It is well-known that holomorphic $1$-forms on the hyperelliptic curve $y^2=h(x)$ of genus $g$ (where $h$ has distinct roots) are of the form $p(x)dx/y$ for $p(x)$ a polynomial of degree at most $g-1$ \cite{Mir}; the only form invariant under the hyperelliptic involution $(x,y)\mapsto(x,-y)$ is the zero form. Hence there is a lift $\tilde{\varphi}$ that acts on $H^0(L)$ as multiplication by $i$.
\end{example}

\vspace{0.5cm}

\section{The Seiberg-Witten equations on the mapping torus $M_{\varphi}$}\label{perturb}

This section contains the key technical work behind the proof of Theorem \ref{main}. Consider automorphism $\varphi$ of $\Sigma$ with $\Sigma/\varphi=\mathbb{P}^1$, leaving invariant a theta characteristic $L$; denote by $\tilde{\varphi}$ a fixed lift, and denote its order by $d$. Choose a $\varphi$-invariant metric on $\Sigma$, and consider the induced metric on the mapping torus $M_{\varphi}$. Under the natural $d$-fold cover $S^1\times\Sigma\rightarrow M_{\varphi}$, the self-conjugate spin$^c$ structure $\spin_L$ pulls back to the unique torsion one, $\spin_0$. It is well known that the unperturbed equations on $S^1\times\Sigma$ do not admit irreducible solutions (see for example \cite{Lop}), hence neither $M_{\varphi}$ admits them. We will then consider perturbations of the form
\begin{equation*}
\bar{\mathcal{L}}=\mathcal{L}-(\delta/2)\|\Psi\|^2+\varepsilon f(B)\text{ with }\delta,\varepsilon>0,
\end{equation*}
where $f$ is induced by a Morse function on the circle $\mathbb{T}$ of reducible solutions on $M_{\varphi}$. We will show that for suitable choices of $f,\delta$ and $\varepsilon$, no irreducible solutions are introduced and one can achieve transversality in the sense of \cite{KM}: the key issue here is to make sure that the perturbed Dirac operators at the reducible critical points has no kernel. 
\par
The perturbation, and the overall strategy of proof, closely follows the analysis in \cite[Chapter 37]{KM} for the case of flat three-manifolds, with a significant complication: their argument strongly relies of the fact that on a flat three-manifold harmonic spinors and $1$-forms are parallel, hence are determined by their value at a given point. This is not true for the manifolds we consider; on the other hand, we will see that the assumption $\Sigma/\varphi=\mathbb{P}^1$ allows to reach the same conclusions via the study of the pulled-back equations on $\Sigma\times S^1$ and the action of $\varphi$.

\vspace{0.3cm}
\subsection{The equations on $S^1\times\Sigma$.}
In this subsection we follow quite closely parts of \cite{Lop}. Consider the manifold $S^1\times\Sigma$ with the product metric. For later conveniences, we will consider the $S^1$ factor to be $\mathbb{R}/2\pi d\mathbb{Z}$, where $d\in\mathbb{N}$ is the order of the lift $\tilde{\varphi}$, with coordinate $t$.

Consider the trivial spin$^c$ structure $\spin_0$ on $\Sigma\times S^1$, denote by $S$ the spinor bundle, and fix a base spin connection $B_0$ arising from a spin structure obtained by taking the product of a spin structure on $\Sigma$ with the Lie structure on $S^1$. When restricted to each slice $\Sigma\times\{t\}$, the spinor bundle decomposes as
\begin{equation}\label{split}
S=S^+\oplus S^-
\end{equation}
line bundles on $\Sigma$ of degree $\pm (g-1)$ as in Section \ref{prelim}. Furthermore, if one identifies $S^+$ as a square root $L$ of the canonical bundle $K$ of $\Sigma$ and $S^{-}$ as $\bar{L}$, the Dirac operator
\begin{equation*}
D_{B_0}:\Gamma(S)\rightarrow \Gamma(S)
\end{equation*}
is given by
\begin{equation*}
D_{B_0}=
\begin{bmatrix}
i\frac{d}{dt} & \sqrt{2}\bardel_{B_0}^*\\
\sqrt{2} \bardel_{B_0}^*&-i\frac{d}{dt}
\end{bmatrix}
\end{equation*}
where we added the dependence of the holomorphic derivative on the spin connection $B_0$ for clarity.
From the three-dimensional viewpoint, the decomposition (\ref{split}) is the one into $\pm i$-eigenspaces of $\rho(dt)$; indeed, if $dx$ and $dy$ are an oriented orthonormal dual basis at a point for $\Sigma$, the Clifford action is given by the Pauli matrices
\begin{equation}\label{pauli}
dt\mapsto\begin{bmatrix}i&0\\0&-i\end{bmatrix},\quad dx\mapsto\begin{bmatrix}0&-1\\1&0\end{bmatrix},\quad dy\mapsto\begin{bmatrix}0&i\\i&0\end{bmatrix}.
\end{equation}
We can write any other spin$^c$ connection as $B=B_0+b$ for some imaginary valued $1$-form $b$; in turn, it will be convenient to decompose
\begin{equation*}
b=a_t+i fdt,
\end{equation*}
where each $a_t$ is an imaginary valued $1$-form on $\Sigma$ and $f$ is a real valued function on $S^1\times \Sigma$. As $F_{B_0^t}=0$, we have
\begin{equation*}
F_{B^t}=2db=2[d_{\Sigma}a_t+(id_{\Sigma}f-\dot{a_t})\wedge dt]
\end{equation*}
where $d_{\Sigma}$ is the exterior differential on $\Sigma$. Hence
\begin{equation*}
\ast F_{B^t}=2[(\ast_{\Sigma}d_{\Sigma}a_t)dt-i\ast_{\Sigma}(d_{\Sigma}f+i\dot{a_t})].
\end{equation*}
Consider a section $\Psi$ of $S\rightarrow Y$ as a pair $(\alpha,\beta)$ of time-dependent sections of $S^{\pm}\rightarrow \Sigma$. We will consider for $\delta\geq 0$ the perturbed Seiberg-Witten equations

\begin{equation}\label{SW}
\begin{aligned}
D_B\Psi=&\delta\Psi\\
\frac{1}{2}\rho(\ast F_{B^t})+(\Psi\Psi^*)_0=&0,
\end{aligned}
\end{equation}
which are the equations for the critical points for the perturbed Chern-Simons-Dirac functional $\mathcal{L}(B,\Psi)-(\delta/2)\|\Psi\|^2$.
In terms of $(\alpha,\beta)$, recall that we have
\begin{equation*}
(\Psi\Psi^*)_0=\begin{bmatrix}
\frac{1}{2}(|\alpha|^2-|\beta|^2)& \alpha\beta^*\\
\beta\alpha^*&\frac{1}{2}(|\beta|^2-|\alpha|^2)
\end{bmatrix}.
\end{equation*}
In order to study (\ref{SW}), we pull-back everything to a $2\pi d$-periodic configuration on $\mathbb{R}\times \Sigma$. We can use the gauge transformation $u=e^{i\int_0^t f}$ to kill the $dt$ component of $b$. Notice that the new configuration is not necessarily $2\pi d$-periodic anymore, but $|\alpha|^2$ and $|\beta|^2$ still are because $u$ is circle valued. The Dirac operator $D_B$ is then just
\begin{equation*}\displaystyle
D_B=
\begin{bmatrix}
i\frac{d}{dt}& \sqrt{2}\bardel_{B_t}^*\\
 \sqrt{2}\bardel_{B_t} &-i\frac{d}{dt}
\end{bmatrix}
\end{equation*}
where $\bardel_{B_t}=\bardel_{B_0}+a_t^{0,1}$. Hence $(\alpha,\beta)$ satisfy the system
\begin{align}
i\dot{\alpha}+\sqrt{2}\bardel_{B_t}^*\beta=\delta\alpha\label{eq1}\\
-i\dot{\beta}+\sqrt{2}\bardel_{B_t}\alpha=\delta\beta.\label{eq2}
\end{align}
Differentiating the second equation in the variable $t$, we then compute
\begin{align}
\ddot{\beta}&=-i\sqrt{2}\frac{d}{dt}\left(\bardel_{B_t}\alpha\right)+i\delta\dot{\beta} \nonumber\\
&=-i\sqrt{2}\left(\bardel_{B_t}\dot{\alpha}+\dot{a_t}^{0,1}\wedge\alpha\right)+i\delta\dot{\beta}\nonumber \\
&=-i\sqrt{2}\left(i\sqrt{2}\bardel_{B_t}\bardel_{B_t}^*\beta-i\delta\bar{\partial}_{B_t}\alpha+\dot{a_t}^{0,1}\wedge\alpha\right)+i\delta\dot{\beta}\label{compdir1}\\
&=2\bardel_{B_t}\bardel_{B_t}^*\beta+|\alpha|^2\beta-\delta\sqrt{2}\bar{\partial}_{B_t}\alpha+i\delta\dot{\beta}\label{compdir2}\\
&=2\bardel_{B_t}\bardel_{B_t}^*\beta+|\alpha|^2\beta-\delta^2\beta \label{compdir3}.
\end{align}
where the steps are:
\begin{itemize}
\item to obtain $(\ref{compdir1})$, we apply $\bardel_{B_t}$ to (\ref{eq1}).
\item to obtain $(\ref{compdir2})$, we use that
\begin{equation*}
\beta\alpha^*=-i\sqrt{2}\dot{a}_t^{0,1}.
\end{equation*}
To see this, notice first that if $\dot{a}_t=i(fdx+gdy)$, then $\dot{a}_t^{0,1}=\frac{-g+if}{2}d\bar{z}$. Furthermore, the component of $\ast_{\Sigma}\dot{a}_t=i(-gdx+fdy)$ acting from $S^+$ to $S^-$ is given by multiplication by $i(-g+if)\frac{d\bar{z}}{\sqrt{2}}$. The anti-diagonal terms of second Seiberg-Witten equation therefore implies that $\beta\alpha^*+i(-g+if)d\bar{z}/\sqrt{2}=0$, hence our claim.
\item to obtain (\ref{compdir3}), we use (\ref{eq2}).
\end{itemize} 
Taking inner products with $\beta$, and integrating over $\{t\}\times\Sigma$, we get
\begin{equation*}
\langle \ddot{\beta},\beta\rangle_{L^2(\Sigma)}=2\|\bardel_{B_t}^*\beta\|_{L^2(\Sigma)}^2+\|\alpha\otimes\beta\|_{L^2(\Sigma)}^2-\delta^2\|\beta\|^2_{L^2(\Sigma)}
\end{equation*}
where $t$ is implicit in the notation for the $L^2$-norms. As 
\begin{equation*}
\frac{1}{2}\frac{d^2}{dt^2}\|\beta\|^2_{L^2(\Sigma)}=\mathrm{Re}\frac{d}{dt}\langle \dot{\beta},\beta\rangle_{L^2(\Sigma)}=\mathrm{Re}\langle \ddot{\beta},\beta\rangle_{L^2(\Sigma)}+\|\dot{\beta}\|_{L^2(\Sigma)}^2,
\end{equation*}
integrating for $t\in[0,2\pi d]$ and using periodicity of $|\beta|$ we get the identity
\begin{equation*}
\|\dot{\beta}\|^2+2\|\bardel_{B_t}^*\beta\|^2+\|\alpha\otimes\beta\|^2=\delta^2\|\beta\|^2
\end{equation*}
for the $L^2$-norms over $S^1\times \Sigma$. A similar computation implies
\begin{equation*}
\|\dot{\alpha}\|^2+2\|\bardel_{B_t}\alpha\|^2+\|\alpha\otimes\beta\|^2=\delta^2\|\alpha\|^2.
\end{equation*}
Using this, we can prove the following well-know fact.

\begin{lemma}\label{redunper}
The only solutions to the Seiberg-Witten equations for $\delta=0$ are the reducible ones.
\end{lemma}
\begin{proof}
In this case, our identities imply that $\alpha$ and $\beta$ are time-independent, and are holomorphic sections of $S^+$ and $S^-$ respectively; because $|\alpha\otimes \beta|=0$, by unique continuation one of them must be identically zero. To conclude, we use the diagonal component of the second Seiberg-Witten equation:
\begin{equation}\label{curv}
i\ast_{\Sigma}d_{\Sigma}a_t+\frac{1}{2}(|\alpha|^2-|\beta|^2)=0.
\end{equation}
Integrating over $\{t\}\times\Sigma$, we obtain
\begin{equation*}
0=\int_{\Sigma}|\alpha|^2-|\beta|^2,
\end{equation*}
so both $\alpha$ and $\beta$ vanish.
\end{proof}

It is important to notice that while the reducible solutions form a smooth manifold (the $(2g+1)$-dimensional torus of flat connections), this is not a Morse-Bott singularity; this is because the Hessian of $\mathcal{L}$ in the normal direction is degenerate exactly where $D_B$ has kernel, and we have the following.
\begin{lemma}\label{theta}
Identify the torus of flat connections as the product of $S^1\times \mathrm{Jac}(\Sigma)$. Then $D_B$ has kernel exactly at $\{0\}\times\Theta$, where $\Theta$ is the theta divisor of $\Sigma$ (which we identify as the divisor in the space of holomorphic line bundles of degree $g-1$ which admit non-zero holomorphic sections).
\end{lemma}
As a consequence, the structure of the singular set is highly dependent on the conformal structure of $\Sigma$.
\begin{proof}
Consider a spin$^c$ connection with $B^t$ flat of the form $B=A+i\lambda dt$ where $B$ is pulled back from a connection on $\Sigma$ and $\lambda\in\mathbb{R}$. The Dirac operator is then given by
\begin{equation*}
D_B=
\begin{bmatrix}
i\frac{d}{dt}-\lambda& \sqrt{2}\bardel_{A}^*.\\
 \sqrt{2}\bardel_{A} &-i\frac{d}{dt}+\lambda
\end{bmatrix}
\end{equation*}
Consider an element $(\alpha,\beta)$ in the kernel of $D_B$. We can expand the spinors in Fourier modes with respect to $t\in\mathbb{R}/2\pi d\mathbb{Z}$:
\begin{equation*}
\alpha=\sum_{n\in\mathbb{Z}} e^{i\frac{n}{d}t}\alpha_n,\qquad\beta=\sum_{n\in\mathbb{Z}} e^{i\frac{n}{d}t}\beta_n,
\end{equation*}
with $\alpha_n,\beta_n$ sections of $S^{\pm}\rightarrow\Sigma$. We get for each $n\in\mathbb{Z}$ the equations
\begin{align*}
-\left(\frac{n}{d}+\lambda\right)\alpha_n+\sqrt{2}\bardel_{A}^*\beta_n&=0\\
 \sqrt{2}\bardel_{A}\alpha_n+\left(\frac{n}{d}+\lambda\right)\beta_n&=0.
\end{align*}
Applying $\sqrt{2}\bardel_{A}$ of the second equation, and substituting the first one in, we see that that
\begin{equation*}
2\bardel_{A}^*\bardel_{A}\alpha_n=-\left(\frac{n}{d}+\lambda\right)^2\alpha_n
\end{equation*}
so that integrating by parts over $\Sigma$, for every $n$
\begin{equation*}
2\|\bardel_{A}\alpha_n\|^2=-\left(\frac{n}{d}+\lambda\right)^2\|\alpha_n\|^2.
\end{equation*}
Similarly, we have the identity
\begin{equation*}
2\|\bardel_{A}^*\beta_n\|^2=-\left(\frac{n}{d}+\lambda\right)^2\|\beta_n\|^2.
\end{equation*}
Suppose that $\alpha_n\neq0$ for some $n$. This implies $\lambda =-n/d$ for some $n\in\mathbb{Z}$, so that $B$ is gauge equivalent to $A$ via the gauge transformation $e^{i\frac{n}{d}t}$ (which is well defined because $t\in\mathbb{R}/2\pi d\mathbb{Z}$). Furthermore, $\bardel_{A}\alpha_n=0$, so that (the line bundle corresponding to the homorphic structure defined by) $A$ belongs to the theta divisor of $\Sigma$. Finally, notice that $\bardel_{A}$ has kernel if and only if $\bardel_{A}^*$ does because they are adjoint to each other and have index zero by Riemann-Roch.
\end{proof}

\vspace{0.3cm}
\subsection{Perturbing the equations on $M_{\varphi}$, part 1.}
Let us begin by discussing the behavior of the perturbed equations on $S^1\times\Sigma$ for $\delta>0$ small. Suppose $(B_i,s_i,\varphi_i)$ is a sequence of irreducible solutions in the blow-up for $\delta_i\rightarrow 0$. Here $\|\varphi_i\|=1$ and we will denote $\varphi_i=(\alpha_i,\beta_i)$. By compactness, such a sequence contains a subsequence converging up to gauge in the blow-up to a solution of the unperturbed equations; this is necessarily reducible by Lemma \ref{redunper}, so that $s_i\rightarrow 0$ and we can write the limit as $(B,0,\varphi)$ with $B^t$ flat, $D_B\varphi=0$ (because of the first Seiberg-Witten equation), and $\|\varphi\|=1$. In particular, $D_B$ has kernel, so the previous lemma implies that $B$ is (up to gauge) the pullback of a connection $A$ on $\Sigma$. Furthermore, the curvature equation (\ref{curv}) is unchanged under our perturbation, and reads in the blow-up as
\begin{equation}\label{curv}
i\ast_{\Sigma}d_{\Sigma}(a_i)_t+\frac{1}{2}s_i^2(|\alpha_i|^2-|\beta_i|^2)=0.
\end{equation}
Because we are assuming $s_i> 0$, integrating we have
\begin{equation*}
\int_{S^1\times\Sigma}|\alpha_i|^2=\int_{S^1\times\Sigma}|\beta_i|^2.
\end{equation*}
In the limit, writing $\varphi=(\alpha,\beta)$, we obtain
\begin{equation*}
\int_{S^1\times\Sigma}|\alpha|^2=\int_{S^1\times\Sigma}|\beta|^2.
\end{equation*}
so that $(\alpha,\beta)$ are \textit{both} non-vanishing (indeed both have $L^2$-norm $1/\sqrt{2}$).
\begin{remark}
Notice that such pairs $(\alpha,\beta)$ do in general exist, so we cannot use this argument to rule out the existence of irreducible solutions on $S^1\times \Sigma$ for $\delta$ small.
\end{remark}
We will now focus on our case of interest, the perturbed equations with $\delta>0$ on the mapping torus $M_{\varphi}$ of $\varphi:\Sigma\rightarrow \Sigma$, where $\Sigma/\varphi=\mathbb{P}^1$. Throughout the section, it will be instructive to have in mind the simple case of the torus, Example \ref{torus}. If the lift $\tilde{\varphi}$ has order $d$, there is a $d$ to $1$ covering map $S^1\times \Sigma\rightarrow M_{\varphi}$, and we denote by $G$ the group of deck transformations (with natural generator $g$); there is a natural induced action of $G$ on the spinor bundle. We will denote by $\tau\in\mathbb{R}/2\pi\mathbb{Z}$ the coordinate in the mapping torus $M_{\varphi}$; the covering map identifies $t$ and $\tau$. Because of our choice of metric, $d\tau$ is the harmonic representative of a generator of $H^1(M_{\varphi};\mathbb{R})\cong \mathbb{R}$.
\par
The pull back of a spin$^c$ structure $\spin_L$ is $\spin_0$, and a configuration on $M_{\varphi}$ pulls-back to a $G$-invariant one on $S^1\times \Sigma$. Gauge transformations $M_{\varphi}\rightarrow S^1$ correspond to $G$-invariant gauge transformations in the cover. We are now ready to prove the following.
\begin{lemma}
Suppose the $\Sigma/\varphi=\mathbb{P}^1$. Then, for sufficiently small $\delta>0$, the perturbed Seiberg-Witten equations do not admit irreducible solutions.
\end{lemma}
\begin{proof}
Suppose the contrary is true. Pulling back to $S^1\times\Sigma$, we obtain from our discussion of compactness above (which holds $G$-equivariantly) that for a connection of the form $B=A_0+i\lambda d\tau$, where $A_0$ is the spin connection of $L$, the operator $D_B$ has kernel containing $G$-invariant configurations $(\alpha,\beta)$ with both spinors non-zero. The only difference is that here we will only consider gauge transformations pulled-back from $M_{\varphi}$, i.e. which are $2\pi$-periodic (rather than just $2\pi d$-periodic); in particular $B$ is not necessarily equivalent to $A_0$.
\par
As in the proof of Lemma \ref{theta}, such a pair $(\alpha,\beta)$ is of the form
\begin{equation*}
(e^{i\frac{n}{d}t}\alpha_*,e^{i\frac{n}{d}t}\beta_*)
\end{equation*}
for $\alpha_*,\beta_*$ sections of $S^{\pm}\rightarrow \Sigma$ in the kernel of the operators $\bardel_{A_0}$ and $\bardel_{A_0}^*$, and some $n\in\mathbb{Z}$. Notice that the gauge equivalence class as configurations in $M_{\varphi}$ is determined by $n$ modulo $d$. If such a configuration is $\tilde{\varphi}$-invariant, we then have that
\begin{equation*}
\tilde{\varphi}^*\alpha_*=e^{2\pi in/d}\alpha_*,\qquad \tilde{\varphi}^*\beta_*=e^{2\pi in/d}\beta_*.
\end{equation*}
As both vectors are non-zero, this implies that $e^{2\pi in/d}$ is an eigenvalue for the action of $\tilde{\varphi}^*$ on both $\ker\bardel_{A_0}=H^0(L)$ and $\ker\bardel_{A_0}^*=H^1(L)$. But the latter is isomorphic to $H^0(L)$ with the conjugate action (Remark \ref{conj}), so that $e^{-2\pi in/d}$ is an eigenvalue of $\tilde{\varphi}^*$ on $H^0(L)$, which contradicts Lemma \ref{noconj}.
\end{proof}

From this, we can also readily determine those $B\in\mathbb{T}$ whose associated Dirac operators $D_B$ on $M_{\varphi}$ has kernel: the $e^{2\pi i n/d}$-eigenspace on $\mathcal{H}^{\pm}$ corresponds to the kernel of the operator associated to $B=A_0-i\frac{n}{d}d\tau$. Lemma \ref{noconj} implies that such kernel only consists either of positive spinors or of negative spinors.
\\
\par
Fix $\delta>0$ sufficiently small. For our purposes, it will be important to also determine the connections $B=A_0+i\lambda d\tau\in\mathbb{T}$ for which the perturbed Dirac equation
\begin{equation}\label{pertdir}
D_B\Phi=\delta\Phi
\end{equation}
on $M_{\varphi}$ has solutions. To do this, after pulling back to $S^1\times \Sigma$, we can proceed as in the proof of Lemma \ref{theta} to obtain
\begin{equation*}
2\bardel_{A_0}^*\bardel_{A_0}\alpha_n=[\delta^2-\left(\frac{n}{d}+\lambda\right)^2]\alpha_n
\end{equation*}
for all $n\in\mathbb{Z}$. Now, $\bardel_{A_0}^*\bardel_{A_0}$ is a non-negative second order elliptic self-adjoint operator, hence it has a spectral gap. Therefore, if $\delta>0$ is small enough, our identity forces
\begin{equation*}
\bardel_{A_0}\alpha_n=0\quad \text{and}\quad [\delta^2-\left(\frac{n}{d}+\lambda\right)^2]\alpha_n=0.
\end{equation*} 
The same holds for the $\beta_n$, and as in Lemma \ref{theta} we conclude
\begin{equation*}
\left(\frac{n}{d}+\lambda+\delta\right)\alpha_n=0\qquad \left(\frac{n}{d}+\lambda-\delta\right)\beta_n=0,
\end{equation*}
so that if for example $\alpha_n\neq0$,  then $\lambda=-\frac{n}{d}-\delta$ and $\beta_n=0$. We have shown the following.
\begin{lemma}\label{pertdir}
Choose $\delta>0$ small. The connections $B\in\mathbb{T}$ such that $D_B\Phi=\delta\Phi$ has non-trivial solutions on $M_\varphi$ are of the form
\begin{equation*}
B=A_0-i\left(\frac{n}{d}\pm\delta\right) d\tau\quad\text{for some }n\in\mathbb{Z},
\end{equation*}
the kernel corresponding to the $e^{2\pi i \frac{n}{d}}$-eigenspace of the action of $\tilde{\varphi}^*$ on the spaces of harmonic spinors $\mathcal{H}^\pm$ respectively.
\end{lemma}

So, roughly speaking, positive harmonic spinors are shifted clockwise while negative harmonic spinors are shifter counterclockwise, see Figure \ref{specflow}.

\begin{figure}[h]
  \centering
\def\svgwidth{\textwidth}
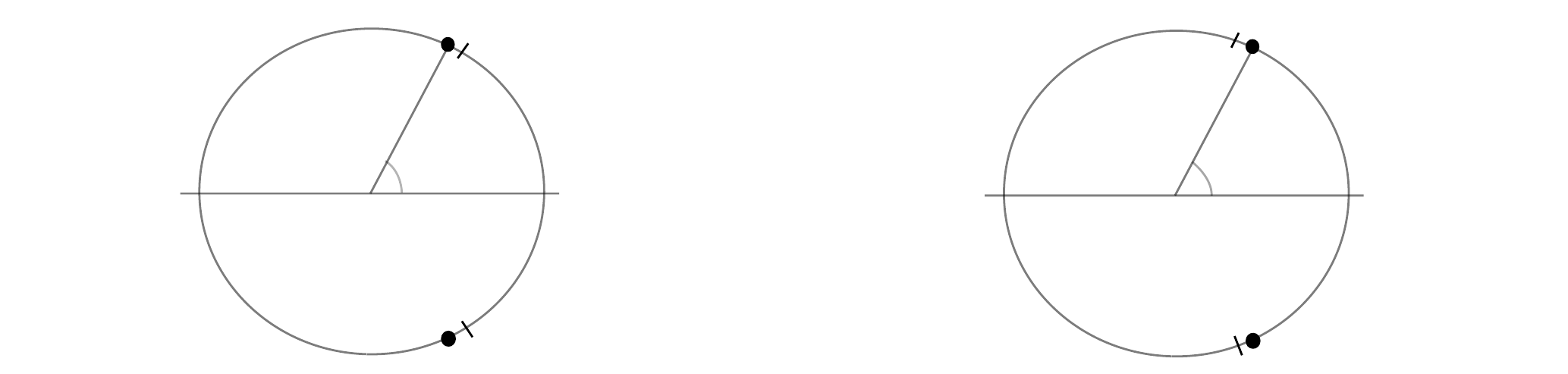
\caption{Consider $0<r<1$. In both pictures, the depicted angle is $\pi r$. The left picture represents the case in which $e^{-\pi ir}$ is an eigenvalue of $\tilde{\varphi}^*$ on $\mathcal{H}^+$, so that the upper (resp. lower) dot corresponds to the $e^{\mp \pi ir}$ eigenspace of the action of $\tilde{\varphi}^*$ on $\mathcal{H}^{\pm}$ and the dashes represent the connections for which $D_B$ has kernel. In the picture on the right, $e^{\pi ir}$ is an eigenvalue of $\tilde{\varphi}^*$ on $\mathcal{H}^+$ instead. Notice that the two cases are mutually exclusive by Lemma \ref{noconj}.}
\label{specflow}
\end{figure}

\subsection{Perturbing the equations on $M_{\varphi}$, part 2.}
The circle of reducible solutions $\mathbb{T}$ is still present after the perturbations we have added. If for none of the $B\in\mathbb{T}$ the Dirac operator $D_B$ on $M_{\varphi}$ has kernel (i.e. by Lemma \ref{pertdir} we have $H^0(L)=0$), we can conclude in the same way as in the case of $S^1\times S^2$ \cite[Chapter $36$]{KM}. Otherwise, in order to proceed choose a Morse function
\begin{equation*}
f:\mathbb{T}\rightarrow \mathbb{R}
\end{equation*}
which has minima exactly at the connections $B\in\mathbb{T}$ for which $D_B$ has kernel (i.e. the union of the spectrum of $\tilde{\varphi}^*$ and its conjugate again by Lemma \ref{pertdir}). We will assume that $f$ is conjugation invariant, so that $\pm1$ are maxima of $f$. For a fixed small $\delta>0$ as above, we consider for $\varepsilon>0$ the perturbed Chern-Simons-Dirac functional
\begin{equation}\label{pertCSD}
\bar{\mathcal{L}}=\mathcal{L}-(\delta/2)\|\Psi\|^2+\varepsilon f(B)
\end{equation}
where, by abuse of notation, $f$ is the function on the configuration space induced by $f$ via first projecting $B$ to the harmonic part of $B-A_0$ and then taking the quotient onto
\begin{equation*}
\mathbb{T}=H^1(M_{\varphi};i\mathbb{R})/H^1(M_{\varphi};2\pi i\mathbb{Z}).
\end{equation*}
Our claim is that for $\varepsilon>0$ small enough, the equations still do not have irreducible solutions. Suppose by contradiction that we have such a sequence $(B_j,s_j,\psi_j)$ with $s_j>0$ for $\varepsilon_j\rightarrow 0$. By our choice of $\delta>0$, we can find a subsequence converging up to gauge to $(B,0,\psi)$, a solution in the blowup to the equations with $\varepsilon=0$. Hence $B\in\mathbb{T}$ and $D_B\psi=\delta\psi$ (looking at the first Seiberg-Witten equation), so that by Lemma \ref{pertdir} $B=B_*+i\lambda d\tau$ for some connection $B_*\in\mathbb{T}$ for which $D_{B_*}$ has kernel  and $\lambda=\mp \delta$ (the sign depending on whether the spinor is positive or negative, respectively). The curvature part of the critical point equations of (\ref{pertCSD}) takes the form
\begin{equation*}
\frac{1}{2}\ast F_{B^t_j}+s^2_j\rho^{-1}(\psi_j\psi_j^*)_0=-i\varepsilon_j \eta(B_j)d\tau.
\end{equation*}
Here $\eta$ is (the pullback to the configuration space of) a $2\pi$-periodic real-valued function on the space of harmonic forms of $M_{\varphi}$. The assumption that $B_*$ is a minimum of $f$ implies that
\begin{equation*}
x\cdot(\eta(B_*+ix d\tau))\geq 0
\end{equation*}
for small values of $x\in\mathbb{R}$ (which we think of as the coordinate of $H^1(M_{\varphi};i\mathbb{R})$ near $B_*$). Using exactness of $F_{B^t_j}$ and $s_j>0$, integrating over $M_{\varphi}$ we see that
\begin{equation*}
\int_{M_{\varphi}}\langle i \lambda_jd\tau,\rho^{-1}(\psi_j\psi_j^*)_0\rangle\leq0
\end{equation*}
where $i\lambda_j dt$ is the harmonic part of $B_j-B_*$, so in the limit
\begin{equation*}
\int_{M_{\varphi}}\langle i \lambda d\tau,\rho^{-1}(\psi\psi^*)_0\rangle\leq0.
\end{equation*}
On the other hand, by Lemma \ref{pertdir} and its proof, $\psi$ is a non-vanishing $D_{B_*}$-harmonic spinor, and satisfies $\rho(i\lambda d\tau) \psi=\delta\psi$; this last equation implies that
\begin{equation*}
\langle i\lambda d\tau,\rho^{-1}(\psi\psi^*)_0\rangle=\delta|\psi|^2/2\geq0,
\end{equation*}
with strict inequality almost everywhere by unique continuation, hence a contradiction.
\\
\par
Hence, for suitable small values of $\varepsilon,\delta>0$, the perturbed functional (\ref{pertCSD}) has no irreducible critical points, and we have a concrete description of the reducible critical points. In particular, the corresponding perturbed Dirac operators have no kernel, so that the critical points in the blow-down are regular. Finally, we can introduce an additional small perturbation to achieve regularity in the blow-up in the sense of \cite{KM} (without introducing irreducible critical points), making in particular the spectrum of the perturbed Dirac operators at the reducible critical points simple; from this, in the next section we will find a concrete description of the Floer chain complex in terms of the action of $\tilde{\varphi}^*$ on $H^0(L)$.

\vspace{0.5cm}
\section{The Floer chain complex}\label{chain}
From the description of the moduli spaces in the previous section, we can directly obtain a concrete expression for the Floer chain complexes as in the following.
Denote as usual by
\begin{equation*}
\T_l=\mathbb{Z}[U,U^{-1}]/U\cdot\mathbb{Z}[U]
\end{equation*}
the standard tower with generators $U^{-n}$ for $n\geq0$, with gradings shifted so that the bottom element has degree $d$.
\begin{thm}Assume $\Sigma/\varphi=\mathbb{P}^1$. If $H^0(L)=0$, then there are choices of metrics and perturbations for which the Floer chain complex of $\HMt_{\bullet}(M_{\varphi},\spin_L)$ is given (up to an overall grading shift) by
\begin{equation*}
\T_{-1}\oplus\T
\end{equation*}
with trivial differential. If $H^0(L)\neq 0$, denote the eigenvalues of $\tilde{\varphi}^*$ by
\begin{equation*}
e^{ \pi i r_1},\dots, e^{ \pi i r_k}
\end{equation*}
with $0<|r_1|<\dots<|r_k|<1$ and multiplicity $m_i$. Then, up to an overall grading shift, the Floer chain complex of $\HMt_{\bullet}(M_{\varphi},\spin_L)$ has the form
\begin{center}
\begin{tikzcd}[column sep=tiny]
  \T\arrow[dr,"\partial_k^r"]& &\T^2\arrow[dr,"\partial_{k-1}^r"]\arrow[dl,"\partial_k^l"]&&\dots&\T^2\arrow[dr,"\partial_i^r"]&&\T^2\arrow[dl,"\partial_i^l"]&&\dots&&\T^2\arrow[dr,"\partial_1^r"]&&\T\arrow[dl,"\partial_1^l"]\\
    &\T^2 & &\T^2&\dots&&\T^2&&&\dots&&&\T^2
\end{tikzcd}
\end{center}
where the towers in the top/bottom row correspond to the maxima/minima of the Morse function on $\mathbb{T}$. Furthermore:
\begin{itemize}
\item  for $1<i<k$, if $r_i>0$, then $\partial^r_i$ is multiplication by $U^{m_i}$ and $\partial^l_i$ is the identity; if $r_i<0$, then $\partial^r_i$ is the identity and $\partial^l_i$ is multiplication by $U^{m_i}$.  
\item If $r_1>0$, then $\partial^r_1$ is multiplication by $U^{m_1}$ and $\partial^l_1$ is the diagonal inclusion; if $r_1<0$, then $\partial^r_1$ is the identity and $\partial^l_1$ is multiplication by $U^{m_1}$ followed by the diagonal inclusion.  
\item If $ r_k>0$, then $\partial^r_k$ is multiplication by $U^{m_k}$ followed by diagonal inclusion and $\partial^l_k$ is identity; if $r_k<0$, then $\partial^r_k$ is the diagonal inclusion and $\partial^l_k$ is multiplication by $U^{m_k}$.  
\end{itemize}
Finally, the gradings of the towers $\T$ are shifted so that all differentials have degree $-1$.
\end{thm}

\begin{proof}
The case $H^0(L)=0$ follows as in the computation for $S^1\times S^2$, \cite[Chapter $36$]{KM}. The case $H^0(L)\neq0$ follows from our description of the moduli spaces in the same way as in the computation for flat manifolds in \cite[Chapter $37$]{KM}. The generators of the Floer chain complex are in bijection with the positive eigenvalues of the perturbed Dirac operators of the critical points of the Morse function $f:\mathbb{T}\rightarrow\mathbb{R}$. Consider now two consecutive critical points $p,q$, respectively a maximum and a minimum of $f$; recall that the minimum is one of the connections $B_*\in\mathbb{T}$ for which $D_{B_*}$ has kernel. Given two stable critical points lying over them with relative grading one, there is exactly one point in the moduli space of trajectories connecting them, so the map is given by
\begin{equation*}
\T\stackrel{U^k\cdot}{\longrightarrow} \T
\end{equation*}
where $k$ is the complex spectral flow of the one parameter family of perturbed Dirac operators (and the gradings are shifted so the map has degree $-1$). The description of Lemma \ref{pertdir} identifies the spectral flow with the kernel of the Dirac operator $D_{B_*}$, hence the description in terms of eigenvalues of $\tilde{\varphi}^*$ (notice that there is spectral flow on exactly one of the two trajectories converging to $B_*$). Finally, orientations are readily determined because we work entirely with reducible moduli spaces.
\end{proof}

From this description it is straightforward to compute the corresponding Floer homology groups; the key observation is that when $H^0(L)\neq0$, in degrees high enough the chain complex looks like
\begin{equation}\label{circchain}
\begin{tikzcd}[column sep=tiny]
  \Z \arrow[dr]&&\Z^2 \arrow[dl]\arrow[dr]&&\cdots&&&\Z^2 \arrow[dr] \arrow[dl]&&\Z \arrow[dl]\\
    & \Z^2&&\Z^2&\cdots&&\Z^2&&\Z^2
\end{tikzcd}
\end{equation}
where all the maps are either the identity or the diagonal inclusion; this is exactly the Morse complex of a circle with a function with exactly $2k$ minima, so the homology is that the circle.
Let us discuss some concrete examples, see Figure \ref{examplesfig}. 
\begin{figure}[h]
  \centering
\def\svgwidth{\textwidth}
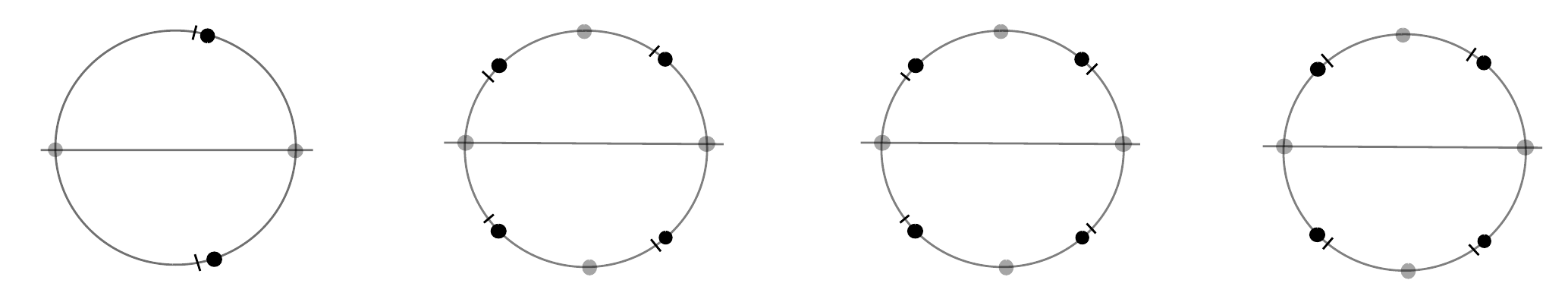
\caption{The pictures, read from left to right, represent the configurations of Examples \ref{ex1}, \ref{ex2}, \ref{ex3}, \ref{ex4}. Here the grey and black dots represent the maxima and minima of the Morse function respectively, and the dashes represents the connections for which the perturbed Dirac equation has kernel.}
\label{examplesfig}
\end{figure} 
\begin{example}\label{ex1}
Suppose there is a single eigenvalue $\lambda_1=e^{\pi i r_1}$ with $r_1>0$ of multiplicity $m_1$. Then, taking into account the gradings, the chain complex is given by
\begin{center}
\begin{tikzcd}
  \T_{-2m_1} \arrow[dr]&&\T\arrow[dl]\\
    & \T^2&
\end{tikzcd}
\end{center} 
and we compute
\begin{equation*}
\HMt_{\bullet}(M_{\varphi},\spin_L)=\T\oplus \T_{1-2m_1}
\end{equation*}
and the action of $\gamma$ is surjective from the second tower onto the first.\end{example}

\begin{remark}
In this example, and all the following ones, the computation for the conjugate spectrum (i.e. changing sign to all the $r_i$) is obtained by reflecting the picture of the chain complex with respect to a vertical line.
\end{remark}

\begin{example}\label{ex2}
Suppose there are two eigenvalues $\lambda_i=e^{\pi i r_i}$ of multiplicity $m_i$ with $r_i>0$. Then the chain complex with gradings is given by

\begin{center}
\begin{tikzcd}
  \T_{-2(m_1+m_2)} \arrow[dr]&&\T^2_{-2m_1}\arrow[dl]\arrow[dr]&&\T\arrow[dl]\\
    & \T^2_{-2m_1}&&\T^2
\end{tikzcd}
\end{center}
and we compute
\begin{equation*}
\HMt_{\bullet}(M_{\varphi},\spin_L)=\T\oplus \T_{1-2(m_1+m_2)}
\end{equation*}
and the action of $\gamma$ is surjective from the second tower onto the first.
\end{example}
\begin{remark}
In general, for homology computations a sequence of consecutive eigenvalues $\lambda_i=e^{\pi i r_i},\dots, \lambda_j=e^{\pi i r_j}$ with $r_i,\dots ,r_j$ of the same sign is equivalent to a single eigenvalue of the same sign and multiplicity $m_i+m_{i+1}+\dots+ m_j$.
\end{remark}

\begin{example}\label{ex3}
Suppose there are two eigenvalues $\lambda_i=e^{\pi i r_i}$ of multiplicity $m_i$ with $r_1<0<r_2$. Then the Floer chain complex is
\begin{center}
\begin{tikzcd}
  \T_{-2m_2} \arrow[dr]&&\T^2\arrow[dl]\arrow[dr]&&\T_{-2m_1}\arrow[dl]\\
    & \T^2&&\T^2
\end{tikzcd}
\end{center}
which has homology
\begin{equation}\label{comp1}
\HMt_{\bullet}(M_{\varphi},\spin_L)=\T\oplus \T_{1-2M}\oplus\mathbb{Z}[U]/U^{m}
\end{equation}
where $M$ and $m$ are respectively the maximum and the minimum of $m_1,m_2$; here the top element of the last summand lies in degree $-1$.
\end{example}

\begin{example}\label{ex4}
Suppose there are two eigenvalues $\lambda_i=e^{\pi i r_i}$ of multiplicity $m_i$ with $r_1>0>r_2$. Then the chain complex is
\begin{center}
\begin{tikzcd}
  \T_{2m_2} \arrow[dr]&&\T^2\arrow[dl]\arrow[dr]&&\T_{2m_1}\arrow[dl]\\
    & \T^2_{2m_2}&&\T^2_{2m_1}
\end{tikzcd}
\end{center}
which has homology
\begin{equation}\label{comp2}
\HMt_{\bullet}(M_{\varphi},\spin_L)=\T\oplus \T_{1-2M}\oplus\mathbb{Z}[U]/U^{m},
\end{equation}
where the bottom elements of the two last summands lie in the same grading. Notice that after changing the orientation we obtain a manifold belonging to Example \ref{ex3}, and indeed (\ref{comp1}) and (\ref{comp2}) are obtained from each other via Poincar\'e duality.
\end{example}

Our description readily adapts to compute the Floer homology with respect to a non-trivial local system. For example, choosing a $1$-cycle $\eta$ with $[\eta]\neq 0\in H^1(M_{\varphi};\mathbb{R})$, we can consider the completed Floer homology $HM_{\bullet}(M_{\varphi},\spin_L;\Gamma_\eta)$ with coefficients in the local system $\Gamma_\eta$ with fiber $\mathbb{R}$; the key point is that in this case the analogue with local coefficients of the chain complex of the circle in Equation $(\ref{circchain})$ has vanishing homology. We obtain:
\begin{itemize}
\item in Examples \ref{ex1} and \ref{ex2}, a cyclic module $\mathbb{R}[U]/U^{m}$ where $m$ is $m_1$ or $m_1+m_2$ respectively;
\item in Example \ref{ex3}, the rank two module $\mathbb{R}[U]/U^{M}\oplus\mathbb{R}[U]/U^{m}$ (the top elements have the same grading);
\item in Example \ref{ex4}, the rank two module $\mathbb{R}[U]/U^{M}\oplus\mathbb{R}[U]/U^{m}$ (the bottom elements have the same grading).
\end{itemize}

With these examples in mind, we are ready to prove the results from the Introduction.

\begin{proof}[Proof of Theorem \ref{thmlocal} and \ref{thmnonlocal}]
Consider the bottom degrees in which the Floer chain complex looks like (\ref{circchain}), and set them to be $0$ and $-1$ respectively. Below this degrees, the Floer chain complex looks like a direct sum of chain complexes of the form
\begin{center}
\begin{tikzcd}[column sep=tiny]
  \Z \arrow[dr]&&\Z \arrow[dl]\arrow[dr]&&\cdots&&&\Z \arrow[dr] \arrow[dl]&&\Z \arrow[dl]\\
    & \Z&&\Z&\cdots&&\Z&&\Z
\end{tikzcd}
\end{center} 
where all maps are isomorphisms, and the top row lies in even absolute $\Z_2$-grading; this has homology $\Z$ lying in the top grading represented by a signed sum of all generators in the top row. In particular the bottom of the tower in homology whose all elements are in odd grading has degree $-1$. The bottom of other the tower is then readily seen to be exactly the lowest grading among the towers corresponding to a maximum of the Morse function, which is exactly the quantity $-2\Delta_L$ in the statement of theorem  \ref{thmnonlocal}.
\par
Furthermore, the conjugation action induces an involution on the set of these chain subcomplexes; such actions has two kind of orbits:
\begin{enumerate}
\item pairs of complexes which are exchanged by the action; these are exactly those that do not involve critical points over the maxima at $\pm1\in\mathbb{T}$;
\item complexes which are sent into themselves by the action; in this case the central copy of $\Z$ lies over $\pm1$.
\end{enumerate}
Looking at the homology with local coefficents as in Theorem \ref{thmlocal}, each such subcomplex contributes a copy of $\mathbb{R}$ to the homology. Furthermore, each pair of `endpoints' of such complexes of contributes an extra dimension to $H^0(L)$ via a spectral flow consideration. This proves that the rank $HM_{\bullet}(M_{\varphi},\spin_L;\Gamma_\eta)$ is exactly $h^0(L)$. Finally, the statements about the ranks as $U$-modules follow from analogous arguments.
\end{proof}

\vspace{0.5cm}
\section{Computations}\label{comp}
In this final section we discuss concrete computations of Floer homology that can be performed via our results by directly computing the action of the automorphism $\varphi$ on the space of holomorphic sections.
\\
\par
We begin with the case of an automorphism $\varphi$ of order $2$, which corresponds to the case of hyperelliptic surfaces; such an automorphism leaves all spin structures invariant. By Lemma \ref{liftorder} (see also Example \ref{torus}), the lift $\tilde{\varphi}$ has order $4$, and (after possibly taking the opposite lift) it acts on the space of harmonic spinors as multiplication by $i$. Via Example \ref{ex1}, Proposition \ref{order2} is then a direct consequence of the computation of the $h^0(L)$ for the corresponding linear systems, which can be found in \cite[Appendix B.$3$]{ACGH} and \cite{BS}; the statement is the following. 
\begin{prop}Consider a hyperelliptic Riemann surface of genus $g$ with hyperelliptic involution $\varphi$, and $[\eta]\neq0$. Among the $2^{2g}$ theta characteristics $L$:
\begin{itemize}
\item exactly $\binom{2g+1}{g}$ have $h^0(L)=0$;
\item for any $1\leq w\leq g$ odd, there are exactly $\binom{2g+2}{g-w}$ theta characteristics $L$ with $h^0(L)=(w+1)/2$.
\end{itemize}
\end{prop}
The proof involves an explicit description of $2^{2g}$ divisors corresponding to the square roots of the canonical bundle in terms of the $2g+2$ fixed points of the hyperelliptic involution.
\\
\par
In light of our results, for a general automorphisms of order $\geq3$ the Floer homology groups will depend not only on the dimension of ${H}^0(L)$ but also on the spectrum of the action of $\tilde{\varphi}^*$. While it is in general quite challenging to compute such data, we can use the Atiyah-Bott $G$-spin theorem \cite{AB} to obtain useful information about it. We start by briefly recalling the geometric setup, which is more natural in the context of the spaces of harmonic spinors $\mathcal{H}^{\pm}$. 
Given a finite order spin automorphism $\varphi$ (with lift $\tilde{\varphi}$) of a $2m$-dimensional spin manifold, we obtain a natural action $\tilde{\varphi}^*$ on the spaces $\mathcal{H}^{\pm}$ of harmonic spinors. Its character is then defined to be
\begin{equation*}
\chi(\tilde{\varphi}^*)=\mathrm{tr}(\tilde{\varphi}^*\lvert_{\mathcal{H}^+})-\mathrm{tr}(\tilde{\varphi}^*\lvert_{\mathcal{H}^-})\in\mathbb{C}.
\end{equation*}
Of course $\chi(\mathrm{id}^*)$ is simply the index of $D^+$, which can be computed using the Atiyah-Singer index theorem. A generalization of this, known as the $G$-spin theorem \cite{AB}, allows to compute $\chi(\tilde{\varphi}^*)$ in terms of the fixed locus of $\varphi$ and the action of $d\varphi$ on the normal bundle of the fixed locus. In general, such a formula is very complicated to use in concrete situations because the local contributions come with a sign (where the lift $\tilde{\varphi}$ comes into play) that is hard to determine. Luckily, in \cite{RRT} the authors worked out exactly the sign contribution in the case of a hyperbolic $2m$-manifold for automorphisms that lift to spin automorphisms of the same order. Via Lemma \ref{liftorder}, we can use this to study a spin action of odd order on a compact Riemann surface. The result is the following, where we consider the unique hyperbolic metric inducing the given conformal structure (notice that uniqueness implies that the metric is $\varphi$-invariant).
\begin{prop}[Theorem $8.1$ of \cite{RRT}]
Consider an automorphism $\varphi$ of odd order $d$ acting on a spin Riemann surface, and consider the lift $\tilde{\varphi}$ with the same order. Suppose that $\varphi$ has fixed points $p_1,\dots,p_m$, and the action of $d\varphi_{p_i}$ on $T_{p_i}$ is given by rotation by $\theta_i$, where $0<|\theta_i|<\pi$. Then
\begin{equation*}
\chi(\tilde{\varphi}^*)=i\sum_i\cos(\tilde{\theta}_i)\csc(\theta_i)=\frac{i}{2}\sum_i\csc(\tilde{\theta}_i),
\end{equation*}
where we denote by $\tilde{\theta}_i$ the square root of $\theta_i$ for which $d\cdot\tilde{\theta}_i\in2\pi\mathbb{Z}$.
\end{prop}
\begin{remark}
The formula implies that $\chi(\tilde{\varphi}^*)\in i\mathbb{R}$; this is consistent with the fact that on a surface the actions of $\tilde{\varphi}^*$ on $\mathcal{H}^{\pm}$ are dual to each other, cf. Remark \ref{conj}.
\end{remark}
\begin{remark}
Notice that the theorem does not assume that $\Sigma/\varphi=\mathbb{P}^1$. In general, given a $\varphi$-invariant spin structure it is already an interesting question to determine whether such spin structure is even or odd (cf. \cite{EOP} for the case in which the quotient is an elliptic curve).
\end{remark}

We will now show how to use this theorem to compute the Floer homology for all automorphism of prime odd order; we begin by proving Proposition \ref{order3} and \ref{order5}.

\begin{proof}[Proof of Proposition \ref{order3}]
If the rotation angle at a fixed point is $\theta=e^{\pm 2\pi i/3}$, the corresponding half-angle is given by $\tilde{\theta}=e^{\mp 2\pi i/3}$; as $\csc(\pm2\pi/3)=\pm2\sqrt{3}/3$, the character of $\tilde{\varphi}^*$ is
\begin{equation*}
\chi(\tilde{\varphi}^*)=\frac{i\sqrt{3}}{3}(n_--n_+).
\end{equation*}
By Lemma \ref{noconj}, $\varphi$ acts on $H^0(L)$ as multiplication by either $e^{\pm 2\pi i/3}$, in which case the character of $\varphi$ is 
\begin{equation*}
\pm(e^{ 2\pi i/3}-e^{- 2\pi i/3})\cdot h^0(L)=\pm i\sqrt{3}\cdot h^0(L).
\end{equation*}
So $H^0(L)$ has dimension $|n_+-n_-|/3$, and the conclusion follows from the computation in Example \ref{ex1}. 
\end{proof}

\begin{proof}[Proof of Proposition \ref{order5}]
In this case the character $\chi(\tilde{\varphi}^*)$ will be an integral combination of
\begin{align*}
e^{2\pi i/5}-e^{-2\pi i/5}=2i\sin(\frac{2\pi}{5})&=i\sqrt{\frac{5+\sqrt{5}}{2}}\\
e^{4\pi i/5}-e^{-4\pi i/5}=2i\sin(\frac{4\pi}{5})&=i\sqrt{\frac{5-\sqrt{5}}{2}}.
\end{align*}
As these are linearly independent over $\mathbb{Q}$, and the spectrum of $\tilde{\varphi}^*$ contains only one of in each pair $e^{\pm 2\pi i/5},e^{\pm 4\pi i/5}$ by Lemma \ref{noconj}, the spectrum is determined uniquely by $\chi(\tilde{\varphi}^*)$. We compute
\begin{align*}
\csc(\widetilde{2\pi/5})/2&=-\frac{1}{5}\cdot2\sin(2\pi/5)-\frac{2}{5}\cdot2\sin(4\pi/5)\\
\csc(\widetilde{4\pi/5})/2&=\frac{2}{5}\cdot2\sin(2\pi/5)-\frac{1}{5}\cdot2\sin(4\pi/5),
\end{align*}
where $\widetilde{2\pi/5}=6\pi/5$ and $\widetilde{4\pi/5}=2\pi/5$. Hence, if the automorphism has $p$ fixed points of rotation $2\pi/5$ and $q$ fixed points of rotation $4\pi/5$ (where we count with signs), then the character is given by
\begin{equation*}
\chi(\tilde{\varphi}^*)=i\left(\frac{-p+2q}{5}\cdot2\sin(\frac{2\pi}{5})+\frac{-2p-q}{5}\cdot2\sin(\frac{4\pi}{5}) \right).
\end{equation*}
We can construct examples in which the action of $\tilde{\varphi}^*$ on $\mathcal{H}^+$ has eigenvalues $e^{2\pi i/5}$ and $e^{4\pi i/5}$ with the desired multiplicities $m,n\in\mathbb{Z}$ (where negative multiplicity corresponds to the conjugate eigenvalue) by taking
\begin{equation*}
p=-m-2n\quad\text{and}\quad q=2m-n.
\end{equation*}
We can write down such automorphism explictly as the natural $\mathbb{Z}/5$-action on a curve of the form $y^5=f(x)$ for a suitable polynomial $f$, following \cite[Chapter III.2]{Mir}. Here, near a root of $f$ of multiplicity $a$ the local model of the singularity of the curve is $z^5=w^a$, and the local rotation is given by $b\cdot 2\pi/5$ where $b$ is the inverse of $a$ modulo $5$. The conclusion of the Proposition follows from the computations in Example \ref{ex3} and \ref{ex4}.
\end{proof}

To conclude, we discuss the proof of Theorem \ref{compprime}, whose proof is a generalization of the method we used in the examples above. The main observation is the following.

\begin{lemma}
Let $p$ an odd prime, and let $\zeta=e^{2\pi i/p}$. Then the numbers $\zeta^i-\zeta^{-i}$ for $i=1,\dots, (p-1)/2$ are linearly independent over $\mathbb{Q}$.
\end{lemma}
\begin{proof}
Suppose there is a non-trivial relation with rational coefficients $a_1,\dots, a_{p-1/2}$. Then
\begin{equation*}
a_1(\zeta-\zeta^{-1})+a_2(\zeta^2-\zeta^{-2})+\dots+a_{(p-1)/2}(\zeta^{(p-1)/2}-\zeta^{-(p-1)/2})=0.
\end{equation*}
Because there is no constant term, this reduces to a polynomial of degree $<p-1$ satisfied by $\zeta$. This is a contradiction because $\zeta$ has minimal polynomial $1+\zeta+\zeta^2+\dots +\zeta^{p-1}$ over $\mathbb{Q}$.
\end{proof}

\begin{proof}[Proof of Theorem \ref{compprime}]
Given an automorphism of odd prime order $p$, we can use the $G$-spin theorem to compute $\chi(\tilde{\varphi}^*)$ in terms of ramification data. Now, $\chi(\tilde{\varphi}^*)$ is an integral linear combination of $\zeta^i-\zeta^{-i}$ for $i=1,\dots (p-1)/2$, and by the previous lemma the coefficients are uniquely determined. Furthermore, no conjugate pairs $\zeta^i,\zeta^{-i}$ appear in the spectrum, so we obtain that $\chi(\tilde{\varphi}^*)$ determines completely the spectrum of the action of $\tilde{\varphi}^*$ on $H^0(L)$; we can then use this to determine the Floer chain complex and its homology as in Section \ref{chain}.
\end{proof}

 \vspace{1cm}

\bibliographystyle{alpha}
\bibliography{biblio}

\end{document}